\documentclass[12pt]{amsart}
\usepackage{amsmath,amsfonts,amsthm,amssymb}
\usepackage{dsfont}
\usepackage{mathtools}
\usepackage[all]{xy}
\usepackage{xcolor}
\SelectTips{cm}{}
\usepackage{fullpage}
\usepackage[pdfencoding=auto]{hyperref}
\usepackage[
backend=biber,
style=alphabetic,
maxnames=99,
sorting=nty
]{biblatex}
\usepackage{comment}

\numberwithin{equation}{section}
\newtheorem{lemma}[equation]{Lemma}
\newtheorem{theorem}[equation]{Theorem}
\newtheorem{proposition}[equation]{Proposition}

\theoremstyle{definition}
\newtheorem{definition}[equation]{Definition}

\newtheorem{eg}[equation]{Example}

\newcommand{\one}{\mathds{1}}

\newcommand{\Ver}{\mathsf{Ver}}%Verlinde category
\renewcommand{\ge}{\geqslant}
\renewcommand{\le}{\leqslant}

\newcommand{\mfrak}{\mathfrak}
\newcommand{\mbb}{\mathbb}
\renewcommand{\one}{\mathds{1}}

\newcommand{\on}{\operatorname}

\newcommand{\wt}{\widetilde}
\newcommand{\ol}{\overline}

\DeclarePairedDelimiter\ang{\langle}{\rangle}
\DeclareMathOperator{\End}{End}
\DeclareMathOperator{\Hom}{Hom}

\DeclareMathOperator{\id}{id}

\DeclareMathOperator{\gr}{gr}

\DeclareMathOperator{\coker}{coker}

\author{Pavel Etingof and Serina Hu}

\address{Department of Mathematics, Massachusetts Institute of
  Technology, Cambridge, MA 02139, USA}
\email{etingof@math.mit.edu}

\address{Department of Mathematics, Massachusetts Institute of
  Technology, Cambridge, MA 02139, USA}
\email{serinahu@mit.edu}

\title{Lie superalgebras in characteristic 2 and mixed characteristic}

\begin{document}

\begin{abstract} We define the notion of a Lie superalgebra over a field $k$ of characteristic $2$ which unifies the two pre-existing ones -- $\Bbb Z/2$-graded Lie algebras with a squaring map and Lie algebras in the Verlinde category ${\rm Ver}_4^+(k)$, and prove the PBW theorem for this notion. We also do the same for the restricted version. Finally, discuss mixed characteristic deformation theory of such Lie superalgebras (for perfect $k$), introducing and studying a natural lift of our notion of Lie superalgebra to characteristic zero -- the notion of a mixed Lie superalgebra over a ramified quadratic extension $R$ of the ring of Witt vectors $W(k)$. 
\end{abstract}

\maketitle

\tableofcontents

\section{Introduction} 

Over a field $k$ of characteristic $\ne 2$, a Lie superalgebra 
is just a Lie algebra $L=L_0\oplus L_1$ in the symmetric tensor category of supervector spaces 
${\rm sVect}_k$, which is the category of $\Bbb Z/2$ graded vector spaces\footnote{In this paper tensor categories, such as ${\rm Vect}_k,{\rm sVect}_k$, ${\rm Ver}_4^+(k)$, etc. are not restricted to finite dimensional objects; rather, they are ind-completions of the rigid subcategories of compact (dualizable) objects.} 
with braiding $v\otimes w\mapsto (-1)^{\deg(v)\deg(w)}w\otimes v$
 (more precisely, in characteristic $3$ one needs an additional condition $[[x,x],x]=0$ 
for $x\in L_1$). Moreover, one defines the enveloping algebra $U(L)$ equipped with a natural increasing filtration, and there is a natural surjective algebra homomorphism $SL=SL_0\otimes \wedge L_1\to {\rm gr}U(L)$, which by the PBW theorem is an isomorphism. Here the exterior algebra of $L_1$ appears because of the relation 
\begin{equation}\label{eq1}
2x^2=[x,x],\ x\in L_1
\end{equation} 
in $U(L)$: in the associated graded algebra this relation becomes 
$2x^2=0$, which implies that $x^2=0$. 

This picture does not naively extend to characteristic $2$, however. 
Indeed, in this case $-1=1$, so the most obvious version of the 
category ${\rm sVect}_k$ is the category of $\Bbb Z/2$-graded vector spaces
with the trivial braiding $v\otimes w\mapsto w\otimes v$. Thus relation \eqref{eq1}
becomes just $[x,x]=0$, $x\in L_1$, so e.g. if $L=L_1$, it is vacuous and we get $U(L)=SL\ne\wedge L$ (as $\wedge L$ is the quotient of $SL$ by the relations $x^2=0$, $x\in L$, so it is smaller than $SL$ if $L\ne 0$). In general, it is reasonable to impose an additional relation $[x,x]=0$ also for $x\in L_0$ (as one does for usual Lie algebras in characteristic $2$), 
which leads to $L$ being just a usual $\Bbb Z/2$-graded Lie algebra, so 
${\rm gr}U(L)=SL=SL_0\otimes SL_1$, again different from $SL_0\otimes \wedge L_1$. 

In the literature, two approaches have been used to address this issue. The more classical approach (see \cite{bouarroudj_vectorial_2020}, \cite{bouarroudj_classification_2023}) 
is to divide relation \eqref{eq1} by $2$. More precisely, in this approach a Lie superalgebra 
over $k$ is defined as a $\Bbb Z/2$-graded Lie algebra $L=L_0\oplus L_1$ 
over $k$ endowed with a quadratic map $Q: L_1\to L_0$ called the {\it squaring map} (morally, $Q(y)=\frac{1}{2}[y,y]$, although this does not make literal sense as both $[y,y]$ and $2$ equal $0$) such that 
$$
Q(y_1+y_2)-Q(y_1)-Q(y_2)=[y_1,y_2]
$$
and 
$$
[Q(y),x]=[y,[y,x]]
$$
for all $x\in L$. Then one defines the super enveloping algebra 
$$
U_{\rm super}(L):=U(L)/(y^2-Q(y),\ y\in L_1)
$$ 
equipped with a natural surjective homomorphism $SL_0\otimes \wedge L_1\to {\rm gr}U_{\rm super}(L)$. The PBW theorem then says that this homomorphism is an isomorphism. 

The second, more ``exotic" approach, which originated in \cite{venkatesh_hilbert_2016} and has been further developed, e.g., in \cite{kaufer_superalgebra_2018}, \cite{hu_supergroups_2024}, \cite{hu_lialg_2025}, is to use a less conventional analog of 
the category ${\rm sVect}_k$ in characteristic $2$ -- the 
symmetric tensor category ${\rm Ver}_4^+(k)$ introduced in \cite{venkatesh_hilbert_2016} (see also \cite{benson_new_2021}). The category ${\rm Ver}_4^+(k)$ is the category of modules 
over the Hopf algebra $H:=k[D]/(D^2)$ with $\Delta(D)=D\otimes 1+1\otimes D$
and braiding defined by the triangular R-matrix $1\otimes 1+D\otimes D$, and, 
as explained in \cite{venkatesh_hilbert_2016}, it is obtained (for perfect $k$) by reduction 
to $k$ of a form of the category of supervector spaces over a ramified 
quadratic extension $R$ of the ring of Witt vectors $W(k)$. Thus objects of ${\rm Ver}_4^+(k)$ do not carry a $\Bbb Z/2$-grading, but the finite dimensional ones are of the form $L=m\one+nP$, where $\one=k$ and $P=H$, which is a reduction to $k$ of a superspace of superdimension $(m+n|n)$. In the second approach, a Lie superalgebra over $k$ is defined as a Lie algebra $L$ in ${\rm Ver}_4^+(k)$ satisfying the PBW condition $[x,x]=0$ if $Dx=0$. Then 
one has the PBW isomorphism $SL\cong {\rm gr}U(L)$ (\cite{kaufer_superalgebra_2018}, \cite{hu_lialg_2025}), and 
if $\lbrace z_1,...,z_m,x_1,...,x_n,y_1,...,y_n\rbrace$ is a basis of $L$ 
such that $Dz_j=0,Dx_i=y_i$ then $SL=k[z_1,...,z_m,x_1,...,x_n]\otimes \wedge (y_1,...,y_n)$, in good agreement with the theory of Lie superalgebras in characteristic $\ne 2$. 

In this paper we show that these two definitions of a Lie superalgebra 
in characteristic $2$ are special cases of a common generalization - the notion of a {\it Lie superalgebra in ${\rm Ver}_4^+(k)$}. In this definition, 
a Lie superalgebra structure is defined on what we call a {\it super-object} of ${\rm Ver}_4^+(k)$, which is an object $L\in {\rm Ver}_4^+(k)$ equipped with a $\Bbb Z/2$-grading 
on the cohomology $\mathcal H(L)$ of $D$ on $L$: $\mathcal H(L)=\mathcal H_0(L)\oplus \mathcal H_1(L)$. Thus finite dimensional super-objects $L=m\one\oplus nP$ are characterized by a triple of integer parameters, $(m_0,m_1,m_2)$, which we call the {\it superdimension of $L$}: $m_j=\dim \mathcal H_j(L)$ (so $m_0+m_1=m$) and $m_2=n$. The first approach is then the special case $m_2=0$ which we call {\it classical} (as then $L\in {\rm Vect}_k$), while the second approach is the special case $m_1=0$, which we call {\it pure}. 
We give examples of Lie superalgebras in ${\rm Ver}_4^+(k)$ and prove a PBW theorem for them. 
We also define the notion of a restricted Lie superalgebra in ${\rm Ver}_4^+(k)$ generalizing the notions of a restricted Lie superalgebra over $k$ from \cite{bouarroudj_classification_2023} and a restricted Lie algebra in 
${\rm Ver}_4^+(k)$ defined in \cite{benson_group_2025} and prove the PBW theorem for these algebras. 

Also, motivated by the fact that ${\rm Ver}_4^+(k)$ is a reduction to $k$ 
of the category of supervector spaces over a ramified $2$-adic ring $R$, we 
consider mixed characteristic deformation theory of Lie superalgebras 
in ${\rm Ver}_4^+(k)$ over $R$. Namely, we define a natural notion of a {\it mixed Lie superalgebra} over $R$, and show that the reduction to $k$ of a mixed Lie superalgebra
is a Lie superalgebra in ${\rm Ver}_4^+(k)$. We then discuss the lifting problem 
for Lie superalgebras in ${\rm Ver}_4^+(k)$, and give a number of examples. 

The organization of the paper is as follows. In Section 2 we discuss the basics of superalgebra
in characteristic $2$ and mixed characteristic, and in particular introduce the categories of mixed superspaces and of super-objects in ${\rm Ver}_4^+(k)$. In Section 3 we define and study Lie superalgebras in characteristic 2 and mixed characteristic, prove our main results, and provide examples. 

{\bf Acknowledgments.} The authors were partially supported by NSF grant DMS-2001318. 

\section{Superalgebra in characteristic $2$ and mixed characteristic}

\subsection{The algebra $H$ of dual numbers and the Verlinde category}

Let $k$ be a field and $H=k[D]/(D^2)$ be the algebra of dual numbers over $k$. Thus an $H$-module is a $k$-vector space $V$ with a linear map $D: V\to V$ such that $D^2=0$. 
We have an additive functor $\mathcal H: H\operatorname{-mod}\to {\rm Vect}_k$ computing the cohomology of $D$. By the Jordan normal form theorem, the only indecomposable $H$-modules are $k$ (the 1-dimensional $H$-module with $D$ acting by $0$) and $P:=H$ (the projective cover of $k$), and any $H$-module $V$ can be uniquely up to isomorphism written as $V=V_k\otimes k\oplus V_P\otimes P$, where $V_k,V_P$ are $k$-vector spaces. Namely, $V_k\cong \mathcal H(V)$ and $V_P\cong {\rm Im}D$. A finite dimensional $H$-module $V$ is thus determined up to isomorphism by 
two numbers $n_1:=\dim V_k=\dim \mathcal H(V)$ 
and $n_2=\dim V_P={\rm rank}(D)$, and we have 
$$
\dim V=n_1+2n_2,\ \dim{\rm Ker}D=n_1+n_2,\ \dim{\rm Im}D=n_2.
$$

If ${\rm char}(k)=2$ then $H$ is a Hopf algebra with $\Delta(D)=D\otimes 1+1\otimes D$, so 
$H\operatorname{-mod}$ is a monoidal category and the functor $\mathcal H$ is monoidal. Moreover, every triangular R-matrix on $H$ defines a symmetric braiding on this category, making the functor $\mathcal H$ symmetric monoidal. 
In particular, we will be interested in the R-matrix 
\begin{equation}\label{rmat}
{\mathcal R}:=1\otimes 1+D\otimes D
\end{equation} 
considered in \cite{venkatesh_hilbert_2016}. 
The category $H\operatorname{-mod}$ equipped with the corresponding symmetric braiding is the {\it Verlinde category} ${\rm Ver}_4^+(k)$ (\cite{benson_new_2021}, Subsection 4.1).

\subsection{Super-structures on $H$-modules} As in the previous subsection, let $k$ be a field and $H:=k[D]/(D^2)$.

\begin{definition} A {\it super-structure} on an $H$-module $V$ is 
a $\Bbb Z/2$-grading on its cohomology, 
$\mathcal H(V)=\mathcal H_0(V)\oplus \mathcal H_1(V)$. 
\end{definition} 

Equivalently, a super-structure of $V$ is a pair of subspaces 
${\rm Im}D\subset V_0,V_1\subset {\rm Ker}D$ such that 
$V_0\cap V_1={\rm Im}D$, $V_0+V_1={\rm Ker}D$. Namely, the spaces $V_j$ are the preimages of $\mathcal H_j(V)$ in ${\rm Ker}D$. 

A {\it super $H$-module} is an $H$-module equipped with a super-structure. Denote the category of super $H$-modules with morphisms being $H$-maps preserving $V_0,V_1$ by $H\operatorname{-smod}$. This is an additive (but not abelian) category equipped with a forgetful functor $${\rm Forg}: H\operatorname{-smod}\to H\operatorname{-mod}.$$ We have the following indecomposable objects in $H\operatorname{-smod}$: 

1) $k_0$ - the space $k$ with $D=0$ and $(k_{0})_0=k,(k_{0})_1=0$. 

2) $k_1$ - the space $k$ with $D=0$ and $(k_{1})_0=0,(k_{1})_1=k$. 

3) $P=H$, $P_0=P_1={\rm Im}D$. 

We have $\dim \Hom(k_i,k_j)=\delta_{ij}$, $\End P=H$, $\dim \Hom(P,k_i)=\dim \Hom(k_i,P)=1$, ${\rm Forg}(k_i)=k$, ${\rm Forg}(P)=P$. 

These are, in fact, the only indecomposable objects. Namely, we have the following lemma. 

\begin{lemma}\label{decompo} Every object $V\in H\operatorname{-smod}$ can be uniquely 
up to isomorphism written as 
\begin{equation}\label{dec}
V\cong V_{k_0}\otimes k_0\oplus V_{k_1}\otimes k_1\oplus V_P\otimes P
\end{equation}
where $V_{k_0},V_{k_1},V_P$ are $k$-vector spaces. 
\end{lemma} 

\begin{proof} Let $V_P\subset V$ be a complement of ${\rm Ker}D$, and let $V_{k_0},V_{k_1}$ be complements of ${\rm Im}D$ in $V_0,V_1$. 
Then \eqref{dec} holds. Moreover, for every decomposition \eqref{dec}, we have 
$V_{k_j}\cong V_j/{\rm Im}D\cong \mathcal H_j(V)$ for $j=0,1$ and $V_P\cong {\rm Im}D$, 
which implies uniqueness.  
\end{proof} 

In particular, we see that a finite dimensional super $H$-module $V$ is determined 
up to isomorphism by three numbers: $m_0=\dim V_{k_0}=\dim \mathcal H_0(V)$, 
$m_1=\dim V_{k_1}=\dim\mathcal H_1(V)$ and $m_2=\dim V_P$. Then 
$$
\dim V=m_0+m_1+2m_2,\ \dim{\rm Ker}D=m_0+m_1+m_2,
$$
$$
 {\rm dim}V_0=m_0+m_2,\ {\rm dim}V_1=m_1+m_2,\ \dim{\rm Im}D=m_2. 
$$

Now suppose that ${\rm char}(k)=2$, so that $H\operatorname{-mod}$ is a monoidal category. 
In this case, $H\operatorname{-smod}$ is also a monoidal category: we define 
$$
\mathcal H_m(V\otimes W):=\oplus_{i+j=m}\mathcal H_i(V)\otimes \mathcal H_j(W).
$$ 
Moreover, any symmetric braiding on $H\operatorname{-mod}$ defines one 
on $H\operatorname{-smod}$. In particular, for the symmetric structure defined by R-matrix \eqref{rmat}, 
 we obtain the additive symmetric monoidal category  of super objects of ${\rm Ver}_4^+(k)$, 
 denoted by ${\rm sVer}_4^+(k)$. 
 
\subsection{Topologically free modules over a complete DVR}
Let $R$ be a complete DVR with maximal ideal $\mathfrak m$, valuation ${\rm val}: R\to \Bbb Z_{\ge 0}$ and residue field $k:=R/\mathfrak m$. For an $R$-module $M$ and $v\in M$, the image of $v$ in $M/\mathfrak{m}M$ will be denoted by $\ol v$. 
An $R$-module $M$ is said to be {\it topologically free} if it is torsion-free, separated and complete in the $\mathfrak m$-adic topology; that is, $M/\mathfrak{m}^nM$ are free 
$R/\mathfrak{m}^n$-modules, and the natural map $M\to \underleftarrow{\lim}M/\mathfrak{m}^nM$ is an isomorphism. So for finitely generated modules, being topologically free is the same as being free. Note that every morphism between topologically free $R$-modules is automatically $\mathfrak m$-adically continuous. We denote the category of topologically free $R$-modules by $TF(R)$. 

Note that the category $TF(R)$ is a symmetric monoidal category under the operation 
of completed tensor product: for $M,N\in TF(R)$, 
$M\otimes N:=\underleftarrow{\lim}(M/\mathfrak{m}^nM\otimes_RN/\mathfrak{m}^nN)$. 
If $M$ or $N$ is finitely generated, this is just the usual tensor product over $R$. 
We have the symmetric monoidal functor of reduction modulo $\mathfrak{m}$, ${\rm Red}: TF(R)\to {\rm Vect}_k$, sending $M$ to $M/\mathfrak{m}M$.

Given a set $I$, denote by ${\rm Fun}_f(I,R)$ the $R$-module of functions $r: I\to R$ such that 
for every $n\in \Bbb N$ there are only finitely many $i\in I$ with ${\rm val}(r(i))\le n$. It is easy to see that ${\rm Fun}_f(I,R)\in TF(R)$, and ${\rm Fun}_f(I,R)/\mathfrak{m}{\rm Fun}_f(I,R)=
{\rm Fun}_f(I,k)$, the space of functions $I\to k$ with finite support. Thus 
for every $k$-vector space $V$ there exists $M\in TF(R)$ with $M/\mathfrak{m}M\cong V$, namely, $M={\rm Fun}_f(I,R)$ where $I$ labels a basis of $V$. Also, given $M\in TF(R)$ and a basis $\lbrace \ol v_i, i\in I\rbrace$ of $M/\mathfrak{m}M$, pick lifts $v_i\in M$ of $\ol v_i$ for $i\in I$. Then $\lbrace v_i, i\in I\rbrace$ is a topological basis of $M$ in the following sense: every element $v\in M$ can be uniquely written as $v=\sum_{i\in I}r(i)v_i$ where $r\in {\rm Fun}_f(I,R)$, and this identifies $M$ with ${\rm Fun}_f(I,R)$ as an $R$-module. It follows that the functor ${\rm Red}$ is full and bijective on isomorphism classes, and every morphism $\alpha$ such that ${\rm Red}(\alpha)$ is an isomorphism is itself an isomorphism (however, ${\rm Red}$ only admits a splitting if $R$ is equal characteristic, i.e., a $k$-algebra). 
In particular, every surjective morphism of topologically free $R$-modules splits. Thus if $N_1,N_2$ 
are topologically free $R$-modules then ${\rm Ext}^1_R(N_1,N_2)=0$.

\subsection{The algebra $S$ and its modules} 
Let $t\in \mathfrak m$ be a uniformizer and $S := R[d]/(d(d-t))$.
It is clear that $S$ does not depend on the choice of $t$ up to a canonical isomorphism: 
if $t_1,t_2$ are two uniformizers then the corresponding $R$-algebras $S_1,S_2$ are identified by renormalizing $d$ by $t_1/t_2$. 

Let $1^*,d^*$ be the $R$-basis of $\Hom_R(S,R)$ dual to the $R$-basis $1,d$ of $S$.  
Given an $R$-module $N$, consider the $S$-modules 
$N\otimes_R S=N\cdot 1\oplus N\cdot d$ and 
$\Hom_R(S,N)=N\otimes_R \Hom(S,R)=N\cdot 1^*\oplus N\cdot d^*$. The map $x\cdot 1+y\cdot d\mapsto y\cdot 1^*+(x+ty)\cdot d^*$ defines 
an isomorphism of $S$-modules $N\otimes_R S\cong \Hom_R(S,N)$. 
 
 Denote the category of $S$-modules which are topologically free as $R$-modules 
by $TF(S,R)$.
Let $M\in TF(S,R)$. Consider the endomorphism
$\ol d$ of $M/tM$ induced by $d$ (so $\ol d^2=0$). Let $K={\rm Ker}\ol d$ 
and $\ol v\in K$. Fix a lift $v\in M$ of $\ol v$. Then $dv\in tM$, so we may consider 
$w:=t^{-1}dv\in M$. We claim that $\ol w\in M/tM$ belongs to $K$ and is independent on the choice of $v$ modulo $J:={\rm Im}\ol d$. Indeed, $dw=t^{-1}d^2v=dv$, so $\ol d\ol w=\ol d\ol v=0$, 
and if $v'=v+tu$ is another lift of $\ol v$ then $w':=t^{-1}dv'=w+du$, so $\ol w'=\ol w+\ol d\ol u$. 
Also if $\ol v\in J$ then $\ol v=\ol d\ol z$ for some $z\in M$, so taking $v:=dz$, we get 
$w=t^{-1}d^2z=dz$, thus $\ol w=d\ol z\in J$.  This implies that the assignment $\ol v\mapsto \ol w$ is a well defined linear endomorphism $\ol e$ of the cohomology space $\mathcal H=K/J$ such that $\ol e^2=\ol e$. It follows that we have a canonical decomposition $\mathcal H=\mathcal H_0\oplus \mathcal H_1$, where $\mathcal H_0:=(1-\ol e)\mathcal H$ and 
$\mathcal H_1:=\ol e\mathcal H$. We thus obtain 

\begin{lemma}\label{func} For $M\in TF(S,R)$, the space $M/tM$ is naturally a super $H$-module with $D=\overline d$, and this defines an additive functor ${\rm Red}: TF(S,R)\to H\operatorname{-smod}$.  
\end{lemma} 
 
Let $R_0,R_1$ be the $S$-modules $S/(d),S/(d-t)$ (so both are $R$ as $R$-modules, with $d$ acting by $0,t$ respectively). We have $\Hom(R_i,R_j)\cong R^{\delta_{ij}}$, 
$\End (S)\cong S\cong R^2$, $\Hom(R_i,S)\cong \Hom(S,R_i)\cong R$. 

\begin{lemma} \label{smod_decomp} Let $M$ be an $S$-module which is topologically free over $R$. Then $M$ admits a unique up to isomorphism decomposition
$$
M \cong M_0\otimes_R R_0 \oplus M_1\otimes_R R_1 \oplus M_2\otimes_R S,
$$
where $M_0,M_1,M_2$ are topologically free $R$-modules. Thus, the functor 
$$
{\rm Red}: TF(S,R)\to H\operatorname{-smod}
$$ 
of Lemma \ref{func} is full and bijective on isomorphism classes, 
and every morphism $\alpha$ such that ${\rm Red}(\alpha)$ is an isomorphism is itself an isomorphism.
\end{lemma}

\begin{proof} We first show the existence of this decomposition. Let $V\subset M/tM$ be a complement of $K:={\rm Ker}\ol d$, and let $\lbrace \overline x_i,i\in I\rbrace$
be a basis of $V$. Let $\overline y_i:=\ol d\overline x_i$. We claim that 
$\lbrace\ol x_i,\ol y_i,i\in I\rbrace$ are linearly independent. Indeed, if $\sum_i b_i\ol y_i=0$ for $b_i\in k$ then 
$\ol d(\sum_i b_i\ol x_i)=0$ so $\sum_i b_i\ol x_i\in V\cap K=0$, i.e. $b_i=0$ for all $i$. 
Hence, if $\sum_i (a_i\ol x_i+b_i\ol y_i)=0$ for $a_i,b_i\in k$ 
then applying $\ol d$ we get $\sum_i a_i\ol y_i=0$, so $a_i=0$ for all 
$i$. Thus $\sum_i b_i\ol y_i=0$, so we also have $b_i=0$ for all $i$, as desired. 

Let $M_2:={\rm Fun}_f(I,R)$. Pick lifts $x_i\in M$ of $\overline x_i$, and let $y_i:=dx_i$ (so $y_i$ are lifts of $\ol y_i$). Since
$\overline x_i,\overline y_i$ are linearly independent, we obtain an inclusion of $S$-modules $M_2\otimes_R S\hookrightarrow M$ with saturated closed image, given by 
$$
r\cdot 1+s\cdot d\mapsto \sum_{i\in I}(r(i)x_i+s(i)y_i).
$$
Therefore, we have a short exact sequence of $S$-modules
    \begin{equation}\label{se1}
    0 \to M_2\otimes_R S \to M \to M' \to 0,
\end{equation}
where $M':=M/(M_2\otimes_R S)$ is a topologically free $R$-module with $\ol d$ 
being the zero map on $M'/tM'$. However, using the Shapiro lemma, 
$$
{\rm Ext}^1_S(M',M_2\otimes_R S)={\rm Ext}^1_S(M',\Hom_R(S,M_2))=
{\rm Ext}^1_R(M',M_2)=0.
$$
Thus the sequence \eqref{se1} splits, and $M\cong M'\oplus M_2\otimes_R S$. 
     
It thus remains to consider the case when $M_2=0$, i.e., $d = 0$ on $M/tM$. Then $d = te$ where $e: M \to M$ and $e^2 = e$. Hence, $M = (1-e)M \oplus eM=(1-e)M\otimes_R R_0\oplus eM\otimes_R R_1$, and we are done.

Finally, let us show the uniqueness of the decomposition, i.e. that the spaces $M_j/tM_j$ are uniquely determined by $M$ up to isomorphism. To this end, note that $M_0/tM_0\cong \mathcal H_0(M/tM)$, $M_1/tM_1\cong \mathcal H_1(M/tM)$, and $M_2/tM_2\cong {\rm Im}\ol d$. 

The last statement follows from Lemma \ref{decompo}. 
 \end{proof}   
  
\subsection{Mixed supervector spaces}
    
From now on suppose that $R$ has characteristic $0$ and perfect residue field $k$ of characteristic $2$, and assume that $2$ generates the ideal $\mathfrak{m}^2$. 
In other words, $R$ is a ramified quadratic extension of the ring of Witt vectors $W(k)$; for example, $R=W(k)[\sqrt{2}]$. Then we can pick a uniformizer $t$ so that $\frac{t^2}{2}\in 1+tR$. Let $Q_R$ be the field of fractions of $R$. 

Consider the group algebra $Q_R[\Bbb Z/2]$ generated by a grouplike element $g$ 
with $g^2=1$, and triangular structure given by the $R$-matrix
$$
\wt{\mathcal R}:=\frac{1}{2}(1\otimes 1+1\otimes g+g\otimes 1-g\otimes g).
$$
The symmetric monoidal category $(Q_R[\Bbb Z/2],\wt{\mathcal R})\operatorname{-mod}$ 
is thus the category of supervector spaces over $Q_R$, 
${\rm sVect}_{Q_R}$.

Let $d:=t\frac{1-g}{2}$, so that $g=1-2t^{-1}d$. Then we have 
$$
d(d-t)=0,\ 
\Delta(d)=d\otimes 1+1\otimes d-2t^{-1}d\otimes d,
$$
and 
$$
\wt{\mathcal R}=1\otimes 1-2t^{-2}d\otimes d.
$$
These formulas define a triangular $R$-Hopf algebra structure on $S$ 
(independent on the choice of $t$ up to isomorphism). Moreover, we have $S\otimes_R Q_R\cong (Q_R[\Bbb Z/2],\wt{\mathcal R})$ whose category of modules is 
${\rm sVect}_{Q_R}$, while $S/tS$ is the triangular Hopf algebra 
$(H,\mathcal R)$ whose category of modules is the category ${\rm Ver}_4^+(k)$. 

\begin{definition} The category $TF(S,R)$ with symmetric monoidal structure defined by $\wt{\mathcal R}$ is called the category of {\it mixed supervector spaces} over $R$ and denoted 
${\rm MixsVect}_R$. 
\end{definition} 

Thus, if $M$ is a mixed supervector space over $R$ then $M\otimes_R Q_R$ 
is a usual supervector space over $Q_R$, and if $M$ is $R$-finite with parameters 
$(m_0,m_1,m_2)$ then the superdimension of $M\otimes_R Q_R$ is $(m_0+m_2,m_1+m_2)$. 
On the other hand, $M/tM\in {\rm sVer}_4^+(k)$, with the same parameters 
$m_0,m_1,m_2$ for $R$-finite $M$. We obtain 

\begin{proposition}\label{rigVer} The functor ${\rm Red}$ of Lemma \ref{func} has a natural structure of a symmetric monoidal functor ${\rm MixsVect}_R\to {\rm sVer}_4^+(k)$. 
\end{proposition} 

\section{Lie superalgebras} 

\subsection{Lie superalgebras in characteristic $2$.} 
Let $k$ be a field of characteristic $2$. Recall the following definition. 

\begin{definition}\label{supe} (\cite{bouarroudj_vectorial_2020} Subsection 1.2.3,  \cite{bouarroudj_classification_2023} Subsection 2.2)  
    A Lie superalgebra over $k$ is a $\mbb{Z}/2$-graded Lie algebra $L= L_0 \oplus L_1$ over $k$ together with a quadratic map $Q: L_1 \to L_0$ such that for $y_1,y_2 \in L_1,x\in L$ we have
$$
        [y_1, y_2] = Q(y_1 + y_2) - Q(y_1) - Q(y_2),\
        [Q(y), x] = [y, [y, x]].
$$
    The super universal enveloping algebra of $L$ is then 
    $$
    U_{{\rm super}}(L) = U(L) / (y^2 - Q(y), y \in L_1).
    $$
\end{definition}

The following theorem is well known, and is a special case of Theorem \ref{pbwth} below. 

\begin{theorem}[PBW theorem]
    Let $L$ be a Lie superalgebra over $k$. Then the natural surjective map 
    $$
    SL_0 \otimes \wedge L_1\to \gr U_{{\rm super}}(L)
    $$
    is an isomorphism. 
\end{theorem}

\subsection{Operadic Lie algebras in ${\rm Ver}_4^+$}

Recall the following definition from \cite{etingof_koszul_2018}. 

\begin{definition} 
    An operadic Lie algebra in a symmetric tensor category $\mathcal C$ is an object $L$ of $\mathcal C$ equipped with a skew-symmetric bracket $\beta: L\otimes L \to L$, i.e.
\begin{equation*}
    \beta \circ (\id + (12)) = 0
\end{equation*}
such that $\beta$ satisfies the Jacobi identity, i.e.
\begin{equation*}
    \beta \circ (\beta \otimes \id) \circ (\id + (123) + (132)) = 0.
\end{equation*}
\end{definition} 

The adjective ``operadic" refers to the fact that we only impose the axioms 
of the Lie operad, so e.g. for $\mathcal{C}={\rm Vect}_k$ with ${\rm char}(k)=2$
we do not require that $[x,x]=0$. 

Thus, an operadic Lie algebra in ${\rm Ver}_4^+(k)$ is an object $L\in {\rm Ver}_4^+(k)$ with bilinear bracket $[,]: L\times L\to L$
such that 
$$
    [x, y] = [y, x] + [y', x'] ,
    $$
    $$
    [[x, y], z] + [[z, x], y] + [[y, z], x] + [[x', y], z'] + [[z', x], y'] + [[y', z], x']=0, 
    $$
    $$
    [x, y]' = [x', y] + [x, y'], 
    $$
    where $x':=Dx$.

\subsection{Operadic mixed Lie superalgebras} 

\begin{definition}\label{operadic} An {\it operadic mixed Lie superalgebra} over $R$ is a Lie algebra in the symmetric monoidal category ${\rm MixsVect}_R$. 
\end{definition}  

Thus, concretely, an operadic mixed Lie superalgebra is a topologically free $R$-module 
$\mfrak{g}$ with an endomorphism $d: \mfrak{g}\to \mfrak{g}$ such that $d(d-t)=0$ and 
a bracket operation $[,]: \mfrak{g} \otimes \mfrak{g} \to \mfrak{g}$ such that
$$
    [x, y] = -[y,x ] + 2t^{-2}[dy, dx],
    $$
    $$ 
    [[x, y], z] + [[z, x], y] + [[y, z], x] -2t^{-2}([[dz, dx], y] + [[dz, x], dy] + [[dy, z], dx]+[[y,dz],dx])=0
 $$   
 and
    $$
d[x, y] = [dx, y] + [x, dy] - 2t^{-1}[dx, dy].
$$

We see that the axioms for an operadic mixed Lie superalgebra reduce modulo $t$ to the axioms for an operadic Lie algebra in ${\rm Ver}_4^+(k)$. Thus we have
\begin{proposition}
    If $\mfrak{g}$ is an operadic mixed Lie superalgebra, then $\overline{\mfrak{g}}:=\mfrak{g}/t\mathfrak{g}$ is an operadic Lie algebra in ${\rm Ver}_4^+(k)$.
\end{proposition}

This motivates the adjective ``operadic" in Definition \ref{operadic}. 

\subsection{The PBW condition} 

Recall (\cite{etingof_koszul_2018}) that the {\it universal enveloping algebra} $U(L)$ 
of an operadic Lie algebra $L$ in $\mathcal C$ 
is the quotient of the tensor algebra $TL$ by the ideal $I_L$ generated by 
the image of the morphism ${\rm id}-\sigma-[,]:L\otimes L\to TL$, 
and that $L$ is called a (genuine) {\it Lie algebra} if the natural map 
$L\to U(L)$ is injective (the PBW condition). In this case, as shown in \cite{etingof_koszul_2018}, the natural surjective map $SL\to {\rm gr}U(L)$ is an isomorphism (the PBW theorem). 

\begin{theorem}\label{PBW} (\cite{kaufer_superalgebra_2018},\cite{hu_lialg_2025}) An operadic Lie algebra $L$ 
in ${\rm Ver}_4^+(k)$ is a Lie algebra if and only if $[x,x]=0$ for every $x\in L$ such that $x'=0$. 
\end{theorem}

Theorem \ref{PBW} implies that if $\mfrak{g}$ is an operadic mixed Lie superalgebra 
then $\ol{\mfrak{g}}=\mathfrak{g}/t\mathfrak{g}$ need not be a Lie algebra in ${\rm Ver}_4^+(k)$ 
(i.e., need not satisfy the PBW condition).
This is demonstrated by the following example. 

\begin{eg}\label{coun}
    Consider the 2-dimensional operadic mixed Lie superalgebra $\mfrak{g}$ with basis $x, y$ with $dx = 0$, $dy = ty$, and nonzero bracket $[y, y] = x$ (it has $m_0 = m_1 = 1, m_2 = 0$). Then $\ol{\mfrak{g}}$ is  not a Lie algebra in ${\rm Ver}_4^+(k)$ (only an operadic one), as $y' = 0$ but $[y, y] \ne 0$.
\end{eg}

\subsection{Mixed Lie superalgebras}

What happens in Example \ref{coun} is the following. Of course, since $R$ has characteristic zero, $\mfrak{g}$ must satisfy the PBW condition, i.e., the natural map $\mfrak{g}\to U(\mfrak{g})$ is injective.\footnote{Here $U(\mfrak{g}):=T\mfrak{g}/I_{\mfrak{g}}$, where the quotient is taken in the category of $R$-modules.} However, this map is not saturated (i.e., its cokernel is not torsion free as an $R$-module), and as a result the reduced map 
$\ol{\mfrak g}\to U(\ol{\mfrak{g}})$ fails to be injective. 
Indeed, $U(\mfrak{g})$ is the algebra generated by $x, y$ with relations $xy = yx$, $2y^2 = x$, so $U(\mfrak{g}) = R[y]$ and the reduced map $\ol{\mfrak{g}}\to U(\ol{\mfrak{g}})$ 
maps $x$ to $0$. 

This motivates the following definition. 

\begin{definition}
    Let $\mfrak{g}$ be an operadic mixed Lie superalgebra. We say that $\mfrak{g}$ is a (genuine) {\it mixed Lie superalgebra} if for all $y \in \mfrak{g}$ such that $dy = ty$, we have $[y, y] \in 2\mfrak{g}$. 
\end{definition}

\begin{proposition} Let $\mathfrak{g}$ be an operadic mixed Lie superalgebra over $R$. Then the following conditions are equivalent. 

(i) $\mfrak{g}$ is a mixed Lie superalgebra;

(ii) The natural injective map $\iota: \mfrak{g}\to U(\mfrak{g})$ is saturated;

(iii) $\ol{\mfrak{g}}$ is a Lie algebra in ${\rm Ver}_4^+(k)$ (i.e. it satisfies the PBW condition).
\end{proposition}

\begin{proof} (ii) $\implies$ (i): Let $y\in \mfrak{g}$, $dy=ty$. Then in $U(\mfrak{g})$ we have a relation $[y,y]=2y^2$, so since $\iota$ is saturated, we must have $[y,y]\in 2\mfrak{g}$.

(iii) $\implies$ (ii): It is easy to see that $U(\ol{\mfrak{g}})=U(\mfrak{g})/tU(\mfrak{g})$. Suppose $\coker \iota$ has torsion. Then there exists an element $x \in \mfrak{g}$ such that $\ol{x} \in \ol{\mfrak{g}}$ is nonzero, but $\iota(x) \in tU(\mfrak{g})$, so $\ol{x} \mapsto 0 \in U(\ol{\mfrak{g}})$. Therefore, $\overline{\mfrak{g}}$ does not satisfy the PBW condition. 

(i) $\implies$ (iii):
Suppose $\mfrak{g}$ is a mixed Lie superalgebra, and let $y_0\in \overline{\mfrak{g}}$ be such that $Dy_0=0$. By the classification of $S$-modules, this means that $y_0$ admits a lift $y=u+v,u,v\in \mfrak{g}$ with $du=0$ and $dv=tv$, so that $[u,u]=0$ and $[v,v]=2w$, $w\in \mathfrak{g}$. Let $u_0,v_0$ be the reductions 
of $u,v$ over $k$. Since $Du_0=0$, we have $[u_0,v_0]=[v_0,u_0]$, so $[y_0,y_0]=[u_0,u_0]+[v_0,v_0]=0$, as desired. 
\end{proof} 

\begin{eg} Let $A$ be an associative algebra in ${\rm MixsVect}_R$. 
Then $A$ has a natural structure of an operadic mixed Lie superalgebra with 
commutator 
$$
[x,y]=x\cdot y-y\cdot x+2t^{-2}dy\cdot dx.
$$
Moreover, it is a mixed Lie superalgebra because if $dy=ty$ then 
$$
[y,y]=2t^{-2}(dy)^2=2y^2\in 2A.
$$
For example, if $A=\underline{\rm End}(V)=V\otimes V^*$ for a finite dimensional 
mixed superspace $V$ 
then the Lie algebra $(A,[,])$ is denoted by $\mathfrak{gl}(V)$. 
If $V=n_0R_0\oplus n_1R_1\oplus n_2S$ 
then we denote $\mathfrak{gl}(V)$ by $\mathfrak{gl}(n_0|n_1|n_2)(R)$. 
It is easy to see that the parameters of this mixed Lie superalgebra are 
$$
(m_0,m_1,m_2)=(n_0^2+n_1^2,2n_0n_1,n_2(2n_0+2n_1+n_2)).
$$ 
\end{eg} 

\subsection{The squaring map of a mixed Lie superalgebra} \begin{definition}
    Let $\mfrak{g}$ be a mixed Lie superalgebra over $R$. Let $\mfrak{g}_0 = {\rm Ker}d$ and $\mfrak{g}_1 = {\rm Ker} (d - t)$. Then we have a natural quadratic map $q: \mfrak{g}_1 \to \mfrak{g}_0$ defined as 
    $$
    q(y) = \frac{1}{2}[y, y].
    $$ 
    \end{definition}
\begin{definition}
    The reduced enveloping algebra of $\mfrak{g}$ is
    $$
    U_{\rm red}(\mfrak{g})= U(\mfrak{g})/(y^2 - q(y), y \in \mfrak{g}_1).
    $$
\end{definition}

Note that elements $y^2 - q(y)$ are $2$-torsion. 

\begin{proposition} The torsion in $U(\mfrak{g})$ is generated by $y^2 - q(y)$, $y\in \mfrak{g}_1$, and $U_{\rm red}(\mfrak{g})$ is a torsion-free $R$-module.
\end{proposition}

\begin{proof}
From the PBW theorem for $\mfrak{g}$, we have $S\mfrak{g} \cong \on{gr} U(\mfrak{g})$. Thus it suffices to prove the statement for abelian $\mfrak{g}$, in which case $U(\mfrak{g})=S\mfrak{g}$. Thus it suffices to prove the statement when $\mfrak{g}$ 
is an indecomposable $S$-module. If $\mfrak{g}=R$ and $d=0$ then $S\mfrak{g}=R[z]$ and there is no torsion. If $\mfrak{g}=R$ with $d=t$ then $S\mfrak{g}=R[z]/(2z^2)$, so the torsion is generated by $z^2$ and the quotient by torsion is $U_{\rm red}(\mfrak{g})=R[z]/(z^2)$. Finally, if $\mfrak{g}=S$ with basis $u,v$ such that $du=v, dv=tv$, then $S\mfrak{g}=R\langle u,v\rangle /(uv-vu+2t^{-1}v^2,2v^2)$. It is easy to see that the torsion in this algebra is generated by $v^2$ and the quotient by this torsion is $U_{\rm red}(\mfrak{g})=R[u,v]/(v^2)$, as desired. 
\end{proof} 

\subsection{Lie superalgebras in ${\rm Ver}_4^+(k)$}
Suppose now that $\mfrak{g}$ is a mixed Lie superalgebra and let $\ol{\mfrak{g}}: = \mfrak{g}/t\mfrak{g}$ be its corresponding Lie algebra in ${\rm Ver}_4^+(k)$. Let 
$\ol{\mfrak{g}}_0: = \mfrak{g}_0/t\mfrak{g}_0$ and 
$\ol{\mfrak{g}}_1: = \mfrak{g}_1/t\mfrak{g}_1$. 
Recall that by Proposition \ref{rigVer}, the subspaces $\ol{\mfrak{g}}_0,\ol{\mfrak{g}}_1$ define a super-structure of $\ol{\mfrak{g}}$, i.e., $\ol{\mfrak{g}}_0\cap \ol{\mfrak{g}}_1={\rm Im}D$
and $\ol{\mfrak{g}}_0+\ol{\mfrak{g}}_1={\rm Ker}D$, and $[\ol{\mfrak{g}}_i,\ol{\mfrak{g}}_j]\subset \ol{\mfrak{g}}_{i+j}$. Also the quadratic map $q$ 
defines a quadratic map 
$$
\ol{q}: \ol{\mfrak{g}}_1 \to \ol{\mfrak{g}}_0
$$
such that 
$$
\ol{q}(x')=[x,x],\ x\in \ol{\mfrak{g}},
$$ 
    $$
    \ol{q}(y_1 + y_2) - \ol{q}(y_1) - \ol{q}(y_2) = [y_1, y_2],
    $$
and 
$$
[\ol{q}(y),x] = [y,[y,x]],\ x\in \ol{\mfrak{g}}.
$$ 

\begin{eg} 1. {\it The classical case.} Suppose that $d(\mfrak{g})\subset t\mfrak{g}$, i.e., $m_2=0$. Then we have $\ol{\mfrak{g}}\in {\rm Vect}_k\subset {\rm Ver}_4^+(k)$, 
so $D=0$ and $\ol{\mfrak{g}}=\ol{\mfrak{g}}_0\oplus \ol{\mfrak{g}}_1$ is an ordinary 
$\Bbb Z/2$-graded Lie algebra. Moreover, we have a quadratic map  
$$
\ol{q}: \ol{\mfrak{g}}_1 \to \ol{\mfrak{g}}_0
$$
such that  
    $$
    \ol{q}(y_1 + y_2) - \ol{q}(y_1) - \ol{q}(y_2) = [y_1, y_2],
    $$
and 
$$
[\ol{q}(y),x] = [y,[y,x]].
$$ 
Thus $\ol{\mathfrak{g}}$ is exactly a Lie superalgebra over $k$ in the sense of Definition \ref{supe}. 

2. {\it The pure case.} Suppose that $m_1=0$, i.e., if $dy=ty$, $y\in \mfrak{g}$ 
then $y=dz$ for some $z\in \mfrak{g}$. Then $\mfrak{g}$ is a usual Lie algebra 
in ${\rm Ver}_4^+(k)$, and $\ol{\mfrak{g}}_0={\rm Ker}D$, $\ol{\mfrak{g}}_1={\rm Im}D$, so 
the data of $\overline q$ is redundant. 
\end{eg} 
 
This motivates the following generalization of Definition \ref{supe}:
\begin{definition}
    A {\it Lie superalgebra in ${\rm Ver}_4^+(k)$} is a Lie algebra $L$ in ${\rm Ver}_4^+(k)$ with a super-structure $L_0, L_1$ such that $[L_i,L_j]\subset L_{i+j}$, 
    and a quadratic map $Q: L_1 \to L_0$ such that 
    $$
    Q(x')=[x,x],\ x\in L
    $$ and 
    \begin{align*}
        Q(y_1 + y_2) - Q(y_1) - Q(y_2) = [y_1, y_2],\ y_1,y_2\in L_1, \\
        [Q(y), x] = [y, [y, x]],\ y \in L_1, x \in L.
    \end{align*}
    We say that $L$ is {\it classical} if $D=0$ (i.e., $L\in {\rm Vect}_k$) and that $L$ is {\it pure} if 
    $L_1={\rm Im}D,L_0={\rm Ker}D$. 
\end{definition}

We thus obtain 

\begin{proposition} If $\mfrak{g}$ is a mixed Lie superalgebra over $R$ then 
$\ol{\mfrak{g}}:=\mfrak{g}/t\mfrak{g}$ has a natural structure of a Lie superalgebra in ${\rm Ver}_4^+(k)$. 
\end{proposition}

In this situation we say that $\mfrak{g}$ is a {\it lift} of $\ol{\mfrak{g}}$ over $R$, and $\ol{\mfrak{g}}$ is the {\it reduction} of $\mfrak{g}$ over $k$.

\begin{eg} Let $\ol{A}$ be an associative algebra in ${\rm sVer}_4^+(k)$. 
Then $\ol{A}$ has a natural structure of a Lie algebra with 
commutator 
$$
[x,y]=x\cdot y-y\cdot x+y'\cdot x'.
$$
Moreover, we have a quadratic map $Q: \ol{A}_1\to \ol{A}_0$ given by $Q(y)=y^2$, 
which gives $\ol{A}$ a natural structure of a Lie superalgebra in ${\rm Ver}_4^+(k)$. 

For example, if $\ol{A}=\underline{\rm End}(V)=V\otimes V^*$ for a finite dimensional 
super-object $V\in {\rm Ver}_4^+(k)$ 
then the Lie algebra $(\ol{A},[,])$ is denoted by $\mathfrak{gl}(V)$. 
If $V=n_0k_0\oplus n_1k_1\oplus n_2P$ 
then we denote $\mathfrak{gl}(V)$ by $\mathfrak{gl}(n_0|n_1|n_2)(k)$. 
It is easy to see that the superdimension of this Lie superalgebra in ${\rm Ver}_4^+(k)$ is 
$$
(m_0,m_1,m_2)=(n_0^2+n_1^2,2n_0n_1,n_2(2n_0+2n_1+n_2)).
$$ 

Moreover, if $A$ is an associative algebra in ${\rm MixsVect}_R$ 
and $\ol{A}:=A/tA$, then $\ol{A}$ is the reduction of $A$ over $k$ as a
mixed Lie superalgebra. For example, $\mathfrak{gl}(n_0|n_1|n_2)(k)$ 
is the reduction over $k$ of $\mathfrak{gl}(n_0|n_1|n_2)(R)$. 
\end{eg} 

\subsection{The alternator of a Lie algebra in ${\rm Ver}_4^+(k)$}

Let $L$ be a Lie algebra in ${\rm Ver}_4^+(k)$. 

\begin{definition} The {\it alternator} of $L$ is the map ${\rm Alt}_L: L/{\rm Ker}D\to {\rm Ker}D$ 
defined by the formula ${\rm Alt}_L(x)=[x,x]$. The {\it reduced alternator} of $L$ is the map $\ol{\rm Alt}_L: L/{\rm Ker}D\to \mathcal H(L)={\rm Ker}D/{\rm Im}D$ induced by ${\rm Alt}_L$. We say that $L$ is {\it alternating} if ${\rm Alt}_L=0$ and {\it weakly alternating} if $\ol{\rm Alt}_L=0$.
\end{definition} 

\begin{lemma} The reduced alternator $\ol{\rm Alt}_L$ is a twisted linear map.
Hence the image of $\ol{\rm Alt}_L$ is a subspace $E(L)$ of $\mathcal H(L)$. 
\end{lemma} 

\begin{proof} It is clear that ${\rm Alt}_L(\lambda x)=\lambda^2 {\rm Alt}_L(x)$, so the same is true for $\ol{\rm Alt}_L$. 
Also 
$$
{\rm Alt}_L(x+y)-{\rm Alt}_L(x)-{\rm Alt}_L(y)=[x,y]+[y,x]=[x',y']=[x',y]'
$$
hence $\ol{\rm Alt}_L$ is additive. 
\end{proof} 

\begin{definition} We say that $L$ is skew-symmetric if ${\rm Alt}_L$ is twisted linear, 
i.e., if $$[x,y]+[y,x]=0$$ (or, equivalently, $[x',y']=0$) for all $x,y\in L$.\footnote{The notions of an alternating and skew-symmetric Lie algebra were introduced in \cite{hu_lialg_2025}.}
\end{definition} 

It is clear that every alternating Lie algebra is skew-symmetric. 

\begin{eg}\label{nwa} (i) Let $L=\one\oplus P$ have basis $x,y,y'$ with $x'=0$ and let 
the only nontrivial commutator of basis elements be $[y,y]=x$. 
Then $L$ is not weakly alternating. 

(ii) It is easy to check that the Lie algebra ${\mathfrak{gl}}(m\one\oplus nP)$ 
is weakly alternating, although not skew-symmetric (hence not alternating) if $n\ne 0$. 
\end{eg} 

\begin{proposition} If $L$ is a weakly alternating Lie superalgebra in ${\rm Ver}_4^+(k)$ 
then $\mathcal H(L)$ has a natural structure of a Lie superalgebra over $k$. 
\end{proposition} 

\begin{proof} Recall that $Q(y')=[y,y]={\rm Alt}_L(y)$. So $L$ is weakly alternating iff 
$Q: {\rm Im}D\to {\rm Im}D$. In this case, if $y\in L_1$ and $z\in {\rm Im}D$ 
then $Q(y+z)=Q(y)+Q(z)+[y,z]$, and the last two summands belong to ${\rm Im}D$, so we obtain a well defined quadratic map $$\overline Q: L_1/{\rm Im}D=\mathcal H^1(L)\to 
L_0/{\rm Im}D=\mathcal H^0(L),$$ which defines a Lie superalgebra structure on 
$\mathcal H(L)$. 
\end{proof} 

\subsection{The super enveloping algebra and the PBW theorem for Lie superalgebras in ${\rm Ver}_4^+(k)$}

\begin{definition}
    Let $L$ be a Lie superalgebra in ${\rm Ver}_4^+$ with squaring map $Q$. 
    The super enveloping algebra $U_{\rm super}(L)$ is
    $$
    U_{\rm super}(L) := U(L) / (y^2 - Q(y), y \in L_1),
    $$
    where $U(L)$ is the enveloping algebra of $L$ as a Lie algebra. 
\end{definition}

Note that if $M\subset L_1$ is a complement to ${\rm Im}D$ then 
$$
    U_{\rm super}(L) = U(L) / (y^2 - Q(y), y \in M),
    $$
    because the relations $y^2 - Q(y)$, $y\in {\rm Im}D$ are already satisfied
    in $U(L)$, and because this relation for $y_1,y_2$ implies one for $y_1+y_2$. 

\begin{eg} Suppose $Q=0$, $[,]=0$. Then 
$U_{\rm super}(L)=SL/(L_1^2)=SL/(M^2)$.
\end{eg}

Thus for a general $L$ we have a natural surjective map 
$SL/(L_1^2)\to {\rm gr}U_{\rm super}(L)$. 

\begin{theorem}\label{pbwth} (PBW)
    The natural map $SL/(L_1^2)\to {\rm gr}U_{\rm super}(L)$ is an isomorphism. 
\end{theorem}

\begin{proof} Let $\lbrace a_1, \dots, a_m\rbrace$ be a basis for ${\rm Im}D$, $\lbrace b_1,...,b_r\rbrace$ be its completion to a basis of $L_0$, $\lbrace c_1,...,c_s\rbrace$ be its completion to a basis of $L_1$, and 
    $h_1,...,h_m\in L$ be such that $h_i'=a_i$. Then all these elements together form a basis of $L$. By the PBW theorem for $L$, a basis for $U(L)$ then consists of ordered monomials in the $a_i,b_j,h_i,c_\ell$ in which $a_i$ occur in powers $\le 1$. Also, the map $L_1\to U(L)$ sending $y$ to $y^2-Q(y)$ is twisted linear and for any $y\in L_1$, $y^2-Q(y)$ is a central element of $U(L)$, as $$[y^2-Q(y),x]=[y,x]y+y[y,x]-[Q(y),x]=[y,[y,x]]-[Q(y),x]=0$$ for all $x\in L$. Thus the elements $z_k:=c_k^2-Q(c_k)$ are central in $U(L)$, and 
    $U_{\rm super}(L)=U(L)/(z_1,...,z_s)$. But the PBW theorem for $L$
    implies that $U(L)$ is a free module over $k[z_1,...,z_s]$ whose basis is formed by monomials in $a_i,b_j,h_i,c_\ell$ in which 
    $a_i$ in $c_\ell$ occur in powers $\le 1$. This implies that these monomials form a basis of $U_{\rm super}(L)$, as claimed. 
\end{proof} 

\subsection{Restricted Lie superalgebras in ${\rm Ver}_4^+(k)$}

Recall that a restricted Lie algebra in characteristic $2$ 
is a Lie algebra $L$ with a quadratic map $Q: L\to L$ such that 
\begin{equation}\label{equaa1}
Q(x+y)-Q(x)-Q(y)=[x,y]
\end{equation}
and 
\begin{equation}\label{equaa2}
[Q(x),y]=[x,[x,y]]
\end{equation}
for all $x,y\in L$. In this case for any $x\in L$ the element $C(x):=x^2-Q(x)$ of $U(L)$ is central, 
the map $x\mapsto C(x)$ is twisted linear, and 
the quotient of $U(L)$ by the relations $C(x)=0$ for $x\in L$ 
is called the restricted enveloping algebra of $L$, denoted by
$U_{\rm res}(L)$. It is well known that 
${\rm gr}U_{\rm res}(L)=SL/(L^2)=\wedge L$ (the restricted PBW theorem). 

This notion extends to the case of Lie superalgebras, 
see \cite{bouarroudj_classification_2023}. Namely, 
a restricted structure on a Lie superalgebra $L=L_0\oplus L_1$
is an extension of $Q: L_1\to L_0$ to a restricted structure on $L$ as a Lie algebra 
compatible with the grading. Such an extension is determined 
by a quadratic map $Q: L_0\to L_0$ which gives a restricted 
structure on $L_0$ such that \eqref{equaa2} holds for $x\in L_0$, $y\in L_1$. Thus the restricted PBW theorem extends straightforwardly to the super-case: $U_{\rm res}(L)=U_{\rm super}(L)/(C(x),x\in L_0)$ and 
${\rm gr}U_{\rm res}(L)=SL/(L^2)=\wedge L$.

The notion of a restricted Lie algebra was recently generalized to Lie algebras in ${\rm Ver}_4^+(k)$ in 
\cite{benson_group_2025} (Definition 18.1). Namely, a restricted Lie algebra in 
${\rm Ver}_4^+(k)$ is a Lie algebra $L$ with a quadratic map 
$Q: {\rm Ker}D\to {\rm Ker}D$ satisfying \eqref{equaa1} for $x,y\in {\rm Ker}D$,\eqref{equaa2} for $x\in {\rm Ker}D,y\in L$, and 
\begin{equation}\label{equaa3}
[x,x]=Q(x'),\ x\in L. 
\end{equation}
Obviously, if $D=0$, this recovers the usual definition of
a restricted Lie algebra. 

We now propose a definition of a {\it restricted Lie superalgebra} 
in ${\rm Ver}_4^+(k)$ which is a common generalization 
of the last two definitions. 

\begin{definition} A restricted Lie superalgebra in ${\rm Ver}_4^+(k)$ is a Lie superalgebra $L$ in ${\rm Ver}_4^+(k)$ together with an extension of $Q: L_1\to L_0$ to a quadratic map 
$Q: {\rm Ker}D\to {\rm Ker}D$ which maps $L_0$ to $L_0$ and satisfies \eqref{equaa1} for $x,y\in L_0$,\eqref{equaa2} for $x\in L_0,y\in L$, and \eqref{equaa3}.
\end{definition}

It is clear that this reduces to the definition 
of \cite{benson_group_2025} when $L$ is pure (i.e., 
$L_0={\rm Ker}D$, $L_1={\rm Im}D$), while 
if $D=0$, i.e., $L=L_0\oplus L_1$ it reduces to the definition 
of a Lie superalgebra over $k$. 

If $L$ is a restricted Lie superalgebra in ${\rm Ver}_4^+(k)$ then for any $x\in {\rm Ker}D$ the element $C(x):=x^2-Q(x)\in U(L)$ is central by \eqref{equaa2}, 
the map $x\mapsto C(x)$ is twisted linear by \eqref{equaa1}, and we may define the restricted enveloping algebra 
$$U_{\rm res}(L):=U(L)/(C(x),x\in {\rm Ker}D)$$ equipped with a natural surjective map 
$$
\xi: SL/(({\rm Ker}D)^2)\to {\rm gr}U_{\rm res}(L). 
$$

The following PBW theorem for restricted Lie superalgebras in ${\rm Ver}_4^+(k)$
is analogous to Theorem \ref{pbwth} and specializes to the PBW theorems for both restricted Lie algebras in ${\rm Ver}_4^+(k)$ and ordinary Lie superalgebras over $k$. 

\begin{theorem} $\xi$ is an isomorphism.  
\end{theorem}

\begin{proof} Let $\lbrace a_1, \dots, a_m\rbrace$ be a basis for ${\rm Im}D$, $\lbrace b_1,...,b_r\rbrace$ be its completion to a basis of ${\rm Ker}D$, and 
    $h_1,...,h_m\in L$ be such that $h_i'=a_i$. Then all these elements together form a basis of $L$. By the PBW theorem for $L$, a basis for $U(L)$ then consists of ordered monomials in the $a_i,b_j,h_i$ in which $a_i$ occur in powers $\le 1$. By \eqref{equaa3}, $C(a_i)=0$ in $U(L)$. Define the elements $z_k:=b_k^2-Q(b_k)$. By 
    \eqref{equaa2},\eqref{equaa3} they are central in $U(L)$ and span $C({\rm Ker}D)$, so 
    $U_{\rm res}(L)=U(L)/(z_1,...,z_s)$. But the PBW theorem for $L$
    implies that $U(L)$ is a free module over $k[z_1,...,z_s]$ whose basis is formed by monomials in $a_i,b_j,h_i$ in which 
    $a_i$ and $b_j$ occur in powers $\le 1$. This implies that these monomials form a basis of $U_{\rm res}(L)$, as claimed. 
\end{proof} 

\subsection{Super-structures on Lie algebras in ${\rm Ver}_4^+(k)$}

Given a Lie algebra $L$ in ${\rm Ver}_4^+(k)$, an interesting problem is to classify 
super-structures on $L$ (i.e., structures of a Lie superalgebra in ${\rm Ver}_4^+(k)$). 
For instance, if $L=m\one\oplus nP$ then for every splitting $m=m_0+m_1$ 
one may ask for a classification of super-structures of superdimension $(m_0,m_1,n)$. 
If $m_1=0$ (the pure case), then there is a unique super-structure (the trivial one), 
so this question is interesting only for $m_1>0$. 

The alternator map provides a useful constraint on super-structures: since for any super-structure, $Q(y')=[y,y]$, $E(L)$ must be contained in $\mathcal H^0(L)$. In particular, 
$m_0\ge \dim E(L)$. For example, this shows that the Lie algebra from Example \ref{nwa}(i)
does not admit nontrivial super-structures, as $m=\dim E(L)=1$ (so we must have $m_0=1$, 
$m_1=0$). 

\subsubsection{Super-structures on $\one + P$} \label{sustr}
Here we classify non-trivial super-structures 
on Lie algebras of the form $L=\one\oplus P$ (i.e., of superdimension $(0,1,1)$)
up to isomorphism using the table in \cite{hu_lialg_2025}, Proposition 4.11. The basis of $L$ we use is $x,y,y'$, with $x'=0$. The super-structure is defined by the value of $Q(x)=\alpha y'$ that satisfies
\begin{align*}
    \alpha[y', y] = [Q(x), y] &= [x, [x, y]] \\
    \alpha[y', x] = [Q(x), x] &= [x, [x, x]] = 0.
\end{align*}
In particular, we see that if $[y', x] \ne 0$, then $\alpha = 0$.

\begin{proposition}\label{sustr1}
The Lie superalgebra structures on $\one + P$ of superdimension $(0, 1, 1)$ up to isomorphism are as follows:

\begin{center}
    \begin{tabular}{| c | c | c | c | c | c |}
        & $[x, y']$ & $[y, y]$ & $[y', y]$ & $[x, y]$ & $\alpha$ \\
        \hline
        1 & $0$ & $0$ & $0$ & $0$ & $0,1$ \\ \hline
        2 & $y'$ & $0$ & $0$ & $y + \lambda y'$ & None \\ \hline
        3 & $0$ & $y'$ & $0$ & $\lambda x$ & $0, 1$ \\ \hline
        4 & $0$ & $y'$ & $0$ & $y'$ & $0$\\ \hline
        5 & $0$ & $x$ & $\lambda x$ & $0$ & None \\ \hline
        6 & $0$ & $x$ & $\lambda x + y'$ & $0$ & None \\ \hline
        7 & $0$ & $0$ & $0$ & $x$ & $0, 1$\\ \hline
        8 & $0$ & $0$ & $0$ & $y'$ & $0,1$\\ \hline
        9 & $0$ & $0$ & $y'$ & $x + y'$ & $0$\\ \hline
        10 & $0$ & $0$ & $y'$ & $\lambda x$ & $0$\\ \hline
        11 & $0$ & $0$ & $x$ & $x + \lambda y'$ &$0$ \\ \hline
        12 & $0$ & $0$ & $x$ & $y'$ & $0$\\ \hline
        13 & $0$ & $0$ & $x$ & $0$ & $0$ \\ \hline
    \end{tabular}
\end{center}

\end{proposition}

The rest of Subsection \ref{sustr} is dedicated to the proof of Proposition \ref{sustr1}. Along the way we also describe the super enveloping algebras of these Lie superalgebras.

Line 1 (abelian Lie algebra). $\alpha=0,1$ (other values can be obtained from $\alpha=1$ by  rescaling $x$). 
The super enveloping algebras are 
\begin{align*}
U_{\alpha=0}=k[x,y,y']/(x^2,(y')^2),\\ 
U_{\alpha=1}=k[x,y]/(x^4)\text{ with }y'=x^2.
\end{align*}

Line 2. None: if $Q(x)=\alpha y'$ then $[Q(x),y]=0$, but 
$[x,[x,y]]=[x,y+\lambda y']=y$. 

Line 3. $\alpha=0,1$ (other values can be obtained from $\alpha=1$ by rescaling $x$).
The super enveloping algebras are 
\begin{align*}
U_{\alpha=0}=k\langle x,y\rangle[y']/(xy-yx-\lambda x,(y')^2-y',x^2),\\ 
U_{\alpha=1}=k\langle x,y\rangle/(xy-yx-\lambda x,x^4-x^2)\text{ with }y'=x^2.
\end{align*}

Line 4. $\alpha=0$ (other values can be obtained by the automorphism 
$x\mapsto x+\beta y'$, $y\mapsto y$).
The super enveloping algebra is 
$$
U=k\langle x,y\rangle[y']/(xy-yx-y',(y')^2-y',x^2).
$$

Lines 5,6. None: these Lie algebras are not weakly alternating. 

Line 7. $\alpha=0,1$  (other values can be obtained from $\alpha=1$ by rescaling $x$).
The super enveloping algebras are
\begin{align*}
U_{\alpha=0}=k\langle x,y\rangle[y']/(xy-yx-x,(y')^2,x^2),\\ 
U_{\alpha=1}=k\langle x,y\rangle/(xy-yx-x,x^4)\text{ with }y'=x^2.
\end{align*}

Line 8. $\alpha=0,1$  (other values can be obtained from $\alpha=1$ by rescaling $y$).
The super enveloping algebras are
\begin{align*}
U_{\alpha=0}=k\langle x,y\rangle[y']/(xy-yx-y',(y')^2,x^2),\\
U_{\alpha=1}=k\langle x,y\rangle/(xy-yx-x^2,x^4)\text{ with }y'=x^2.
\end{align*}

Line 9. $\alpha=0$. Indeed, we have $[x,[x,y]]=[x,x+y']=0$, so $[Q(x),y]=0$. 
The super enveloping algebra is 
$$
U=k\langle x,y,y'\rangle/(xy-yx-x-y',yy'-y'y-y',(y')^2,x^2).
$$

Line 10. $\alpha=0$. Indeed, we have $[x,[x,y]]=[x,\lambda x]=0$, so $[Q(x),y]=0$. 
The super enveloping algebra is 
$$
U=k\langle x,y,y'\rangle/(xy-yx-\lambda x,yy'-y'y-y',(y')^2,x^2).
$$

Line 11. $\alpha=0$.  Indeed, we have $[x,[x,y]]=[x,x+\lambda y']=0$, so $[Q(x),y]=0$.
The super enveloping algebra is 
$$
U=k\langle x,y,y'\rangle/(xy-yx-x-\lambda y',yy'-y'y-x,(y')^2,x^2).
$$

Line 12. $\alpha=0$. Indeed, we have $[x,[x,y]]=[x,y']=0$, so $[Q(x),y]=0$.
The super enveloping algebra is 
$$
U=k\langle x,y,y'\rangle/(xy-yx-y',yy'-y'y-x,(y')^2,x^2).
$$

Line 13. $\alpha=0$. Indeed, we have $[x,[x,y]]=0$, so $[Q(x),y]=0$. 
The super enveloping algebra is 
$$
U=k\langle x,y,y'\rangle/(xy-yx-y',yy'-y'y-x,(y')^2,x^2).
$$
Note that in all these instances the PBW theorem can be checked easily by hand.

\subsubsection{Super-structures on $2 \cdot \one + P$} 
We also outline how one would compute Lie superstructures on $L := 2 \cdot \one + P$ and work them out in the first Lie algebra structure on $2 \cdot \one + P$ described in \cite{hu_lialg_2025} (Proposition 4.16). Let $L$ have basis $x, y, v, v'$ with $x' = y' = 0$. Either $m_1 = m_0 = 1$ or $m_1 = 2, m_0 = 0$. 

If $m_0 = 0$, $m_1 = 2$, then $L$ must be weakly alternating for it to have a super-structure. Write $Q(x) = \alpha v', Q(y) = \beta v'$. Thus
\begin{align*}
    \alpha[v', x] &= \beta[v', y] = 0 \\
    \alpha[v', y] &= [x, [x, y]] \\
    \alpha[v', v] &= [x, [x, v]] \\
    \beta[v', x] &= [y, [y, x]] \\
    \beta[v', v] &= [y, [y, v]].
\end{align*}

If $m_1 = m_0 = 1$, say that $L_1 = \ang{z, v'}$, $L_0 = \ang{w, v'}$, where $z, w$ are a basis for $2 \cdot \one$.
We have to define $Q(z) = \gamma w + \delta v'$ such that
\begin{align*}
    [Q(z),x] &= [z,[z,x]] \\ 
[Q(z),y] &= [z,[z,y]] \\
[Q(z),v] &= [z,[z,v]].
\end{align*}
We can reduce to 3 cases for $z, w$:
\begin{enumerate}
\item $z = x, w = y + \alpha x$, in which case $\gamma, \delta$ satisfy 
\begin{align*}
    \gamma[y,x]+\delta[v',x]          &= 0 \\
\gamma \alpha[x,y]+\delta[v',y]        &= [x,[x,y]] \\
\gamma[y,v]+\gamma\alpha[x,v]+\delta[v',v] &= [x,[x,v]].
\end{align*}
\item $z = y + \alpha x, w = y + \beta x$, $\alpha \ne \beta$, in which case $\gamma, \delta$ satisfy
\begin{align*}
    \gamma[y,x]+\delta[v',x]          &= [y+\alpha x,[y,x]] \\
\gamma\beta[x,y]+\delta[v',y]        &= \alpha[y+\alpha x,[x,y]] \\
\gamma[y,v]+\gamma\beta[x,v]+\delta[v',v] &= [y+\alpha x,[y+\alpha x,v]]. 
\end{align*}
\item $z = y + \alpha x, w = x$ in which case $\gamma, \delta$ satisfy 
\begin{align*}
    \delta[v',x]        &= [y+\alpha x,[y,x]] \\
\gamma[x,y]+\delta[v',y] &= \alpha [y+\alpha x,[x,y]] \\
\gamma[x,v]+\delta[v',v] &= [y+\alpha x,[y+\alpha x,v]]. 
\end{align*}
\end{enumerate}

\begin{proposition}
Let $L$ be the Lie algebra on $2 \cdot \one + P$ with nonzero brackets $[v, y] = \lambda x$ and $[x, y] = v'$. It has four Lie superalgebra structures with superdimension $(0, 2, 1)$ up to isomorphism, where $Q(x) = \alpha v', Q(y) = \beta v'$ for $\alpha, \beta \in {0, 1}$, when $\lambda = 0$. It has two Lie superalgebra structures with superdimension $(1, 1, 1)$ up to isomorphism when $\lambda \neq 0$, namely $L_1 = \ang{x, v'}$, $L_0 = \ang{y, v'}$, $Q(x) = 0$ or $v'$. It has two Lie superalgebra structures with superdimension $(1, 1, 1)$ up to isomorphism when $\lambda = 0$, namely 
\begin{enumerate}
    \item $L_1 = \ang{y, v'}$, $L_0 = \ang{y + x, v'}$, $Q(y) = 0$ or $v'$
    \item $L_1 = \ang{y, v'}$, $L_0 = \ang{x, v'}$, $Q(y) = 0$ or $v'$.
\end{enumerate}
\end{proposition}
\begin{proof}
Suppose $m_0 = 0, m_1 = 2$. Then the only constraint on $Q$ is $\beta[v', v] = [y, [y, v]] = \lambda v'$, so $\lambda = 0$. Then by scaling $y$, we have $\beta = 0, 1$. Since $x, v'$ must scale together to preserve $[x, y] = v'$ and $Q$ is quadratic, we can also scale $\alpha$ to make sure that $\alpha = 0$ or $\alpha=1$. So the super-enveloping algebra is of the form
\begin{equation*}
    U_{\beta, \alpha} = k\ang{x, y}[v, v']/(xy-yx-v', (v')^2,x^2-\alpha v', y^2 - \beta v'), \beta, \alpha \in \{0, 1\}.
\end{equation*}

Now suppose $m_0 = m_1 = 1$. Notice that $y$ can be translated by multiples of $x$ while preserving the bracket. Hence we can assume $\alpha = 0$. 
\begin{enumerate}
    \item In this case, $z = x, w = y$. Then $\gamma v' = 0$, so $\gamma = 0$. So $Q(x) = \delta v'$. We can scale so $\delta = 0, 1$. So
    \begin{equation*}
U_{\lambda, \delta}  = k\ang{x,y,v}[v']/(xy-yx-v',xv-vx,yv-vy-\lambda x,(v')^2,x^2-\delta v').
    \end{equation*}
    \item Here we may scale $x$ so that $\beta = 1$, so $z = y, w = y + x$. Then $\gamma v' = 0$, so $\gamma = 0$. The last condition also implies that $\gamma \lambda x = \lambda v'$, so $\lambda = 0$. Then $Q(y) = \delta v'$ and we can scale so $\delta = 0, 1$. The universal enveloping superalgebra is then
    \begin{equation*}
        U_\delta=k\ang{x,y}[v,v']/(xy-yx-v',(v')^2,y^2-x^2-\delta v').
    \end{equation*}
    \item Here $z = y, w = x$. By the same logic as the previous case, $Q(y) = \delta v'$ for $\delta = 0, 1$, and the universal enveloping superalgebra is the same as above.
\end{enumerate}
\end{proof}

\subsection{Lifting Lie superalgebras in $\Ver_4^+$ to mixed Lie superalgebras}

Let $\ol{\mfrak{g}}$ be a Lie superalgebra in $\Ver_4^+(k)$ and let $K = Q_R$.

\begin{definition}
We say that $\ol{\mfrak{g}}$ lifts to a mixed Lie superalgebra $\mfrak{g}$ if $\mfrak{g}/t\mfrak{g} \cong \ol{\mfrak{g}}$. We say that $\ol{\mfrak{g}}$  lifts to a given Lie superalgebra $\widetilde{\mfrak{g}}$ over $Q_R$ if $\ol{\mfrak{g}}$ has a lift $\mfrak{g}$ over $R$ such that $\mfrak{g} \otimes_R Q_R\cong \widetilde{\mfrak{g}}$.
\end{definition}

We first determine if $\ol{\mfrak{g}}$ can admit a lift to a Lie algebra over $R/t^2 R$. Similar to \ref{smod_decomp}, we can check that $d$ admits a lift. Suppose $\ol{\mfrak{g}}$ is a Lie superalgebra in $\Ver_4^+$ that has underlying object $m \cdot \one \oplus n P$ and type $(m_0, m_1, n)$, $m_0 + m_1 = m$.
Then we can check if it admits a lift over $R/t^2R$ as follows. Choose a basis $u_i, v_j, z_\ell, w_\ell$ of $m_0 R_0 \oplus m_1 R_1 \oplus n S$ with $du_i = 0$, $dv_j = tv_j$, $dz_\ell = w_\ell$, $dw_\ell = tw_\ell$. Let $C$ be the collection of structure constants of $\mfrak{g}_0$ over $R/tR$ in this basis and $\wt{C}$ be any lift of $C$ over $R/t^2 R$. Recall that in a mixed Lie superalgebra, we require
\begin{align}
     [x, y] + [y,x ] - [dy, dx]=0 \label{ss}\\
    [[x, y], z] + [[z, x], y] + [[y, z], x] + [[dx, y], dz] + [[dz, x], dy] + [[dy, z], dx]=0 \label{jacobi} \\
     d[x, y] - [dx, y] - [x, dy] + t[dx, dy]=0 \label{daction} \\
    [y, y] - 2 Q(y)=0, y \in \ker (d - t). \label{yyeq}
\end{align}
Let
\begin{align*}
    D_2(\wt{C}) = \frac{\text{LHS of \ref{ss}}}{t} \pmod{t} \\
    D_3(\wt{C}) = \frac{\text{LHS of \ref{jacobi}}}{t} \pmod{t} \\
    E_2(\wt{C}) = \frac{\text{LHS of \ref{daction}}}{t} \pmod{t}
\end{align*}
The lift exists iff there exists a correction $\wt{C}' = \wt{C} + t F$ which kills $D_2$, $E_2$, and $D_3$. This is a system of inhomogeneous linear equations on $F$, and the lift exists iff this system has a solution.\footnote{Note that $[y, y] = 0 \pmod{t^2}$ for $y \in \ker (d - t)$, so we always have $[y, y ] -2 Q(y) \equiv 0 \pmod{t^2}$ and we don't need to worry about equation \eqref{yyeq}.} If $F_1$ is a solution of this system then $F_2$ is a solution iff $F_1-F_2$ satisfies the corresponding homogeneous system, i.e., is a cocycle in a suitable sense. Furthermore, there is an obvious notion of equivalence of two solutions $F_1,F_2$: they are equivalent if there is an isomorphism between the corresponding lifts $\mathfrak{g}_1,\mathfrak{g}_2$ over $R/t^2R$ which is the identity modulo $t$. This happens iff $F_1-F_2$ is a coboundary in a suitable sense, so the set of all lifts up to equivalence is a (possibly empty) torsor over the corresponding cohomology group (quotient of cocycles by coboundaries), as is natural to expect when studying mixed characteristic deformations of any algebraic object. However, here we will not work out this cohomology theory in detail. 

If we wish to lift directly to $R$ instead of $R/t^2 R$, instead we would need to find a correction that satisfies \ref{ss}, \eqref{jacobi}, \eqref{daction}, and \eqref{yyeq}, which is a system of quadratic equations on $F$. 

Let $P$ have basis $x, x'$. Recall that we have three Lie algebra structures on $P$: abelian, $[x', x] = x'$, $[x,x]=0$, and $[x, x] = x'$, $[x',x]=0$. Let $K^{(1, 1)}$ be the supervector space of dimension $(1, 1)$ spanned by $p, q$ with $p$ even and $q$ odd.

\begin{proposition}\label{2dim}
     The abelian Lie algebra on $P$ admits a lift to the abelian Lie superalgebra on $K^{(1, 1)}$; the Lie algebra with nonzero bracket $[x, x'] = x'$ admits a lift to the Lie superalgebra with $[p, q] = q$, $[q,q]=0$; and the Lie algebra with nonzero bracket $[x, x] = x'$ admits a lift to the Lie superalgebra with $[q, q] = p$, $[p,q]=0$.
\end{proposition}
\begin{proof}
     When lifting to $R$, we must have $m_0 = m_1 = 0$ and $m_2 = 1$, so there is no need to consider superstructures. Let $z, w$ be a basis for $S$ with $dz = w, dw = t w$; then it is easy to see that in each case, the naive lift $x \mapsto z, x' \mapsto w$ satisfies the necessary identities. 
\end{proof}

Let $\one + P$ have basis $x, y, y'$. Let $L=L(\lambda)$ be the Lie algebra on $\one + P$ with the only possibly nonzero brackets of basis vectors $[y, y] = y'$, $[x, y] = \lambda x$ with $\lambda \in k$. Let $m_0 = 1$.

\begin{proposition}
     If $\lambda\ne 0$ then $L$ does not admit a lift even to $R/t^2 R$. However, if $\lambda = 0$, so $[y, y]$ is the only nonzero bracket, then $L$ lifts to the Lie superalgebra on $K^{(2, 1)}$ with basis $p,r$ (even), $q$ (odd) and the only nonzero bracket $[q, q] = p$.
\end{proposition}
\begin{proof}
    Let $u, z, w$ be a basis for $R_0 \oplus S$.
    The commutation relations in a lift of $L$ have the form $[u, u] = tE_{uu}$, $[z, z] = w + tE_{zz}$, $[w, w] = tE_{ww}$, $[u, z] = \lambda u + tE_{uz}$, $[u, w] = tE_{uz}$, $[w, z] = tE_{wz}$. 

    Note that because $[w, w] = t^2[z, z]$, we get that $E_{ww} \equiv 0 \pmod{t}$. Also by skew-symmetry, $E_{zw} \equiv - E_{wz}$, $E_{uz} \equiv -E_{zu}$, $E_{uw} \equiv -E_{wu}$ mod $t$.

    The equality $d[z, z] = 0$ implies that $t(w + E_{zz}^zw) \equiv 0 \pmod{t^2}$, so we need $E_{zz}^z \equiv 1 \pmod{t}$. The equality $d[u, z] = [u, w]$ implies that $tE_{uz}^z w \equiv t E_{uw} \pmod{t^2}$, so $E_{uz}^z \equiv E_{uw}^w \pmod{t}$ and $E_{uw}^u \equiv E_{uw}^z \equiv 0 \pmod{t}$. The equality $d[w, z] = t[w, z] - [w, w]$ implies that $E_{wz}^z \equiv 0 \pmod{t}$, since the RHS is $0$ mod $t^2$.

    To check the Jacobi identities:
    from $[[z, z], z] + [[w, z], w] = 0$ we get
    $$
        0 = [w, z] + t[E_{zz}, z] + t[E_{wz}, w] 
        = tE_{wz} + tE_{zz}^z[z, z] + tE_{zz}^u[u, z] 
        = tE_{wz} + tw + tE_{zz}^u \lambda u,
    $$
    so $E_{wz}^u \equiv \lambda E_{zz}^u \pmod{t}$, $E_{wz}^w \equiv 1$, and $E_{wz}^z \equiv 0$.

    Since $[[u, z], w] + [[w, u], z] + [[z, w], u] + t[[w, u], w] = 0$, we get
    $$
        0 = [\lambda u + tE_{uz}, w] + t[E_{uw}, z] + t[E_{wz}, u] 
        = \lambda t E_{uw},
    $$
    since $E_{uw}^u = E_{uw}^z \equiv 0$ while the only nonzero bracket with $u$ is $[z, u]$ but $E_{wz}^z = 0$.
    Hence $E_{uw} \equiv 0 \pmod{t}$ as $\lambda \ne 0$.

    Now consider the relation $[[z, z], u] + [[w, u], w] = 0$. Since $E_{uw} \equiv 0$, this reduces to $[[z, z], u] = 0$. Then $[w + tE_{zz}, u] = 0$, and we know that $[u, w]$ and $[u, u]$ are multiples of $t$ while $E_{zz}^z \equiv 1$, so this is equivalent to $t(\lambda u) \equiv 0 \pmod{t^2}$, which is impossible unless $\lambda = 0$.

    However, if $\lambda = 0$ and $[y, y]$ is the only nonzero bracket, then $L$ is a direct sum of a 1-dimensional 
    commutative Lie algebra and a 2-dimensional Lie superalgebra, 
    so the statement follows from Proposition \ref{2dim}.
\end{proof}
\begin{proposition}
    Let $L$ be the Lie algebra on $\one + P$ with nonzero brackets $[y, y] = y'$, $[x, y] = \lambda x$ with $\lambda \ne 0$. Let $m_1 = 1$ with $\alpha = 0$, so $Q(x) = 0$ and $Q(y) = y'$. Then $L$ does not admit a lift to $R/t^2 R$. 
\end{proposition}   
\begin{proof}
    Let $v, z, w$ be a basis for $R_1 \oplus S$.
    The commutation relations in a lift of $L$ have the form $[v, v] = tE_{vv}$, $[z, z] = w + tE_{zz}$, $[w, w] = tE_{ww}$, $[v, z] = \lambda v + tE_{vz}$, $[v, w] = tE_{vw}$, $[w, z] = tE_{wz}$. 

    Note that because $[w, w] = t^2[z, z]$, we get that $E_{ww} \equiv 0 \pmod{t}$. Also by skew-symmetry, $E_{zw} \equiv - E_{wz}$, $E_{vz} \equiv -E_{zv}$, $E_{vw} \equiv -E_{wv}$ mod $t$.

    The equality $d[z, z] = 0$ implies that $t(w + E_{zz}^zw) \equiv 0 \pmod{t^2}$, so we need $E_{zz}^z \equiv 1 \pmod{t}$. The equality $d[v, z] = t[v, z] - [v, w]$ implies that $tE_{vz}^z w \equiv t E_{vw} \pmod{t^2}$, so $E_{vz}^z \equiv E_{vw}^w \pmod{t}$ and $E_{vw}^v \equiv E_{vw}^z \equiv 0 \pmod{t}$. The equality $d[w, z] = t[w, z] - [w, w]$ implies that $E_{wz}^z \equiv 0 \pmod{t}$,  since the RHS is $0$ mod $t^2$.

    Now we check the Jacobi identities.
    Since $[[z, z], z] + [[w, z], w] = 0$, we have 
    $$
        0 = [w, z] + t[E_{zz}, z] + t[E_{wz}, w] 
        = tE_{wz} + tE_{zz}^z[z, z] + tE_{zz}^v[v, z] 
        = tE_{wz} + tw + tE_{zz}^v \lambda v,
    $$
    so $E_{wz}^v \equiv \lambda E_{zz}^v \pmod{t}$, $E_{wz}^w \equiv 1$, and $E_{wz}^z \equiv 0$.

    Since $[[v, z], w] + [[w, v], z] + [[z, w], v] + t[[w, v], w] + t[[w, w], v] = 0$, we have
    \begin{align*}
        0 &= [\lambda v + tE_{vz}, w] + t[E_{vw}, z] + t[E_{wz}, v] \\
        &= \lambda t E_{vw}^w w + tE_{vz}^v[v, w] + tE_{vz}^w[w, w] + tE^w_{vw}[w, z] + tE^z_{vz}[z, w],
    \end{align*}
    but since $[v, w]$, $[w, w]$, and $[w, z]$ are all $0 \pmod{t}$, we have $E_{vw}^w = 0 \pmod{t}$. Therefore, $E_{vw} = 0 \pmod{t}$.

    Now consider $[[z, z], v] + [[v, z], z] + [[z, v], z] + t[[v, z], w] + [[w, v], w] + t[[w, z], v] = 0$. Since $E_{vw} \equiv 0$ and $[v, z] = [z, v] \pmod{t}$, this reduces to $[[z, z], v] = 0$. Then $[w + tE_{zz}, v] = 0$, which simplifies to $E_{zz}^z[z, v]=0$. Hence, $\lambda v = 0 \pmod{t}$, contradiction.
\end{proof}

\printbibliography

% \begin{thebibliography}{99999}
% \bibitem[1]{bouarroudj_vectorial_2020}
% \bibitem[2]{bouarroudj_classification_2023}
% \bibitem[3]{venkatesh_hilbert_2016}
% \bibitem[4]{benson_new_2021} Benson, Etingof, Ostrik
% \bibitem[5]{Eti} Etingof, PBW
% \bibitem[6]{Hu1} Serina Hu - 1
% \bibitem[7]{Hu2} Serina Hu - 2
% \bibitem[8]{Kau} Kaufer
% \end{thebibliography} 

\end{document}